\documentclass[11pt, reqno]{amsart}

\usepackage{amsmath,amssymb,amsthm,amsfonts,verbatim}
\usepackage{microtype}
\usepackage[all,2cell]{xy}
\usepackage{mathtools}
\usepackage{graphicx}
\usepackage{pinlabel}
\usepackage{hyperref}
\usepackage{mathrsfs}
\usepackage{color}
\usepackage[dvipsnames]{xcolor}
\usepackage{enumitem}
\usepackage{cite}
\usepackage{soul}

\CompileMatrices

\usepackage[top=1.2in,bottom=1.2in,left=1in,right=1in]{geometry}

\usepackage{hyperref} 
\hypersetup{          
	colorlinks=true, breaklinks, linkcolor=[RGB]{51 102 204}, filecolor=Orchid, urlcolor=[RGB]{51 102 204},
	citecolor=Orchid, linktoc=all, }
\usepackage[nameinlink]{cleveref}

\theoremstyle{plain}
\newtheorem{theorem}{Theorem}[section]
\newtheorem{maintheorem}{Theorem}

\newtheorem{maincor}[maintheorem]{Corollary}

\newtheorem{proposition}[theorem]{Proposition}
\newtheorem{lemma}[theorem]{Lemma}

\newtheorem{corollary}[theorem]{Corollary}

\theoremstyle{definition}
\newtheorem{definition}[theorem]{Definition}

\newcommand{\nc}{\newcommand}
\nc{\dmo}{\DeclareMathOperator}

\nc{\cA}{\mathcal{A}}
\nc{\cB}{\mathcal{B}}
\nc{\sB}{\mathscr{B}}
\nc{\C}{\mathbb{C}}
\nc{\cC}{\mathcal{C}}
\nc{\sC}{\mathscr{C}}
\nc{\BB}{\mathbb{B}}
\nc{\LL}{\mathcal{L}}
\nc{\bd}{\mathbf{d}}
\nc{\DD}{\mathbb{D}}
\nc{\cD}{\mathcal{D}}
\nc{\sD}{\mathscr{D}}
\nc{\bF}{\mathbb{F}}
\nc{\cF}{\mathcal{F}}
\nc{\cG}{\mathcal{G}}
\nc{\cI}{\mathcal{I}}
\nc{\cH}{\mathcal{H}}
\nc{\cK}{\mathcal{K}}
\nc{\cL}{\mathcal{L}}
\nc{\cM}{\mathcal{M}}
\nc{\bM}{\mathbf{M}}
\nc{\N}{\mathbb{N}}
\nc{\cN}{\mathcal{N}}
\nc{\cO}{\mathcal{O}}
\nc{\bp}{\mathbf{p}}
\nc{\bP}{\mathbb{P}}
\renewcommand{\P}{\mathbb{P}}
\nc{\cP}{\mathcal{P}}
\nc{\Q}{\mathbb{Q}}
\nc{\R}{\mathbb{R}}
\nc{\cS}{\mathcal{S}}
\nc{\cT}{\mathcal T}
\nc{\cU}{\mathcal U}
\nc{\bV}{\mathbb V}
\nc{\cX}{\mathcal{X}}
\nc{\cY}{\mathcal{Y}}
\nc{\Z}{\mathbb{Z}}

\nc{\disk}{\mathbb{D}}
\nc{\hyp}{\mathbb{H}}
\nc{\CP}{\mathbb{CP}}
\nc{\RP}{\mathbb{RP}}
\dmo{\Mod}{Mod}
\dmo{\PMod}{PMod}
\dmo{\LMod}{LMod}
\dmo{\Diff}{Diff}
\dmo{\Homeo}{Homeo}
\dmo{\dist}{dist}
\dmo\BDiff{BDiff}
\dmo\SO{SO}
\dmo\Sp{Sp}
\dmo\SL{SL}
\dmo\Out{Out}
\dmo\Aut{Aut}
\dmo\Inn{Inn}
\dmo\GL{GL}
\dmo\PGL{PGL}
\dmo\Gr{Gr}
\dmo\PSL{PSL}
\dmo\BHomeo{BHomeo}
\dmo\EHomeo{EHomeo}
\dmo\EDiff{EDiff}
\dmo\Disc{Disc}
\dmo\Aff{Aff}

\dmo\Teich{Teich}
\dmo{\Vol}{Vol}
\dmo{\Push}{Push}
\dmo{\Conf}{Conf}
\dmo{\PConf}{PConf}
\dmo{\PB}{PB}
\dmo{\Jac}{Jac}
\dmo{\Pic}{Pic}
\dmo{\Arf}{Arf}
\dmo{\Gal}{Gal}
\dmo{\SU}{SU}
\dmo{\spin}{spin}
\dmo{\even}{even}
\dmo{\odd}{odd}
\dmo{\orb}{orb}
\dmo{\pr}{pr}
\dmo{\lab}{lab}
\dmo{\Sym}{Sym}
\dmo{\Rad}{Rad}
\dmo{\Ind}{Ind}
\dmo{\Div}{Div}
\dmo{\Hur}{Hur}

\dmo\Hom{Hom}
\dmo\rank{rank}
\dmo\sig{sig}
\nc{\Span}[1]{\operatorname{Span}(#1)}
\dmo{\vcd}{vcd}
\dmo{\codim}{codim}
\dmo{\Res}{Res}
\dmo{\Ann}{Ann}
\dmo{\Int}{Int}
\dmo{\lcm}{lcm}
\dmo{\ab}{ab}
\dmo{\opp}{op}
\dmo{\End}{End}
\dmo{\Stab}{Stab}
\dmo{\id}{id}
\dmo{\im}{im}
\dmo\Fix{Fix}
\dmo{\Isom}{Isom}

\nc{\pair}[1]{\ensuremath{\left\langle #1 \right\rangle}}
\nc{\abs}[1]{\ensuremath{\left| #1 \right|}}
\nc{\action}{\circlearrowright}
\nc{\norm}[1]{\ensuremath{\left | \left | #1 \right | \right |}}
\nc{\abcd}[4]{\ensuremath{\left(\begin{array}{cc} #1 & #2 \\ #3 & #4 \end{array}\right)}}
\nc{\into}\hookrightarrow
\newcommand{\onto}{\twoheadrightarrow}
\nc{\normal}{\vartriangleleft}

\renewcommand{\bar}{\overline}

\newcommand{\sslash}{\mathbin{/\mkern-6mu/}}
\renewcommand{\epsilon}{\varepsilon}
\renewcommand{\tilde}{\widetilde}
\renewcommand{\le}{\leqslant}
\renewcommand{\ge}{\geqslant}

\nc{\margin}[1]{\marginpar{\scriptsize #1}}
\nc{\para}[1]{\medskip\noindent\textbf{#1.}}
\definecolor{myblue}{RGB}{102,153, 255}
\definecolor{myred}{RGB}{204,0,0}
\definecolor{mygreen}{RGB}{0,204,0}
\definecolor{myorange}{RGB}{255,102,0}
\definecolor{mypurple}{RGB}{138,43,226}
\nc{\red}[1]{\textcolor{myred}{#1}}
\nc{\blue}[1]{\textcolor{myblue}{#1}}

\dmo{\Sing}{Sing}
\dmo{\SpSym}{SpSym}
\dmo{\SpPic}{SpPic}
\dmo{\Br}{Br}
\dmo{\SBr}{SBr}
\dmo{\type}{type}

\date{September 29, 2025}
\title[$\pi_1$ of low-codimension strata]{Fundamental groups of strata of abelian differentials of low codimension}
\author{Nick Salter}
\address{NS: Department of Mathematics, University of Notre Dame, 255 Hurley Building, Notre Dame, IN 46556}
\email{nsalter@nd.edu}

\begin{document}
\begin{abstract}
    We show that for a stratum of projectivized abelian differentials with sufficiently many simple zeroes, the inclusion into the appropriate moduli space of pointed curves induces an injection at the level of orbifold fundamental group, thereby confirming part of a 1997 conjecture of Kontsevich-Zorich. Combined with previous work of the author with Calderon, this describes the orbifold fundamental groups of these strata as framed mapping class groups. Our approach is algebro-geometric in nature, and is based on Shimada's techniques for computing fundamental groups via morphisms of varieties that behave like fiber bundles away from loci of high codimension.
\end{abstract}

\maketitle

\section{Introduction}

Let $\cM_{g,p}$ denote the moduli space of $p$-pointed compact Riemann surfaces of genus $g$, and let $(\P)\Omega_{g,p}$ denote the (projectivized) Hodge bundle, with fiber over $C \in \cM_{g,p}$ the (projective) space of holomorphic $1$-forms on $C$. There is a {\em stratification} of $(\P)\Omega_{g,p}$ according to the multiplicities of the zeroes of the differentials. Let $\kappa = \{k_1, \dots, k_{p}\}$ be a partition of $2g-2$. Then $\cH_\kappa \subset \Omega_g$ denotes the corresponding stratum of Abelian differentials $(C,\omega)$, where $C \in \cM_g$ is a Riemann surface (with no distinguished points), and $\omega$ is a holomorphic $1$-form on $C$ whose zeroes have multiplicities $k_1,\dots, k_p$; let $(\P)\cH_\kappa$ denote the corresponding locus in $(\P)\Omega_g$. Assume $k_1 \ge \dots \ge k_n > k_{n+1} = \dots = k_p = 1$ (allowing $n=0$), and define the quantity 
\[
d(\kappa) = \sum_{i = 1}^n k_i;
\]
note this only counts the degrees of the zeroes of multiplicity $\ge 2$. There is an inclusion
\[
\rho: \P\cH_\kappa \into \cM_{g,p} 
\]
sending a projectivized differential $(C, [\omega])$ to the underlying curve $C$ marked with the divisor $(\omega)$. Our main theorem asserts that when the stratum includes sufficiently many simple zeroes, this is injective at the level of orbifold fundamental groups.

\begin{maintheorem}\label{theorem:main}
For $g\ge d(\kappa) + \max\{n+4, 7\}$, $\rho_*: \pi_1^{orb}(\P\cH_\kappa) \to \pi_1^{orb}(\cM_{g,p})$ is injective.  
\end{maintheorem}

The codimension of $\P\cH_\kappa \subset \P\Omega_g$ is computed to be the quantity $c(\kappa) = \sum_{i=1}^p (k_i-1)$, so that the relation $d(\kappa) = c(\kappa) + n$ holds. For fixed genus $g$, the bounds in \Cref{theorem:main} thus enforce that the stratum $\P\cH_\kappa$ be of sufficiently low codimension in $\P\Omega_g$, or equivalently that a sufficient proportion of the zeroes of differentials in the stratum be simple.

The image of $\rho_*$ in $\Mod_{g,p}$ was previously described by the author and Calderon in \cite{strata3}. Recall that a {\em framing} of a $d$-manifold is a trivialization of its tangent bundle, or equivalently $d$ everywhere linearly independent vector fields. When $S_{g,p}$ is an oriented surface of genus $g$ with $p$ punctures, any non-vanishing vector field or $1$-form induces a well-defined isotopy class of framing. The mapping class group $\Mod_{g,p}$ (the group of orientation-preserving diffeomorphisms of $S_{g,p}$ up to isotopy) acts on the set of isotopy classes of framings. For such a class $\phi$, there is an associated stabilizer subgroup known as the {\em framed mapping class group} $\Mod_{g,p}[\phi]$. 

For $(C,[\omega])$ a projectivized differential, any representative $\omega$ therefore determines a framing of $C \setminus (\omega) \cong S_{g,p}$; different choices of representatives $\lambda \omega$ for $\lambda \in \C^*$ yield isotopic framings. It follows that the monodromy $\rho'_*: \pi_1^{orb}(\P\cH_\kappa) \to \Mod_{g,p}$ of the universal family of translation surfaces over $\P\cH_\kappa$ is valued in the associated framed mapping class group $\Mod_{g,p}[\phi_{\cH_\kappa}]$. As $\pi_1^{orb}(\cM_{g,p}) = \Mod_{g,p}$, the maps $\rho_*$ and $\rho'_*$ are identified. In \cite[Theorem A]{strata3}, it was shown that for $g \ge 5$, $\im(\rho'_*)= \Mod_{g,p}[\phi_{\cH_\kappa}]$, for any (non-hyperelliptic) stratum of Abelian differentials.

\begin{maincor}\label{maincor}
    For $g \ge d(\kappa) + \max\{n+4, 7\}$, there is an isomorphism
\[
\pi_1^{orb}(\P \cH_\kappa) \cong \Mod_{g,p}[\phi_{\cH_\kappa}].
\]
\end{maincor}

\para{The conjecture of Kontsevich-Zorich} In their 1997 paper \cite{kzconj}, Kontsevich and Zorich pose the problem of understanding the homotopy type of strata, motivated by their evident close connection with moduli spaces of Riemann surfaces and the mapping class group. In the form it appears in the paper \cite[page 7]{kzconj}, it is conjectured that each connected component of a stratum\footnote{In the seminal paper \cite{kz}, Kontsevich-Zorich classify the path components of every stratum $\cH_\kappa$. Following their classification, any stratum with at least one simple zero is connected, so the issue of distinct path components does not arise in this paper.} should be an orbifold $K(\pi,1)$ space with orbifold fundamental group ``commensurable to some mapping class group''. \Cref{maincor} shows that for strata of sufficiently low codimension, this conjectural description of the orbifold fundamental group is correct.

\para{Codimension: low vs. high} The theory of {\em period coordinates} shows that the dimension of $\cH_\kappa$ coincides with the relative Betti number $\dim(H_1(C,(\omega);\C)) = 2g+p-1$; the dimension of $\P\cH_\kappa$ thus drops to $2g+p-2$. As touched on briefly above, the codimension of $\P\cH_\kappa \subset \P\Omega_g$ can be computed from this to be $c(\kappa) = \sum_{i = 1}^n (k_i-1) = d(\kappa) - n$.

It is worth observing that projectivized strata $\P\cH_\kappa$ in genus $g$ have a wide range of dimensions, from $4g-4$ for the principal stratum (differentials with all simple zeroes), down to $2g-1$ for the minimal stratum (differentials with a single zero of order $2g-2$). Sitting in the middle is the dimension $3g -3$ of the moduli space $\cM_g$ itself. Thus, as the partition coarsens and multiplicities increase, the strata go from having large-dimensional fibers over (most of) $\cM_g$, to being high-codimension loci inside $\cM_g$, supported only on curves with Weierstrass points of exceptional multiplicity. It is reasonable to expect that topological properties of strata (e.g. asphericality) could vary as one moves between low and high codimension settings.

The results of this paper take place at the top end of the dimension spectrum. A rough heuristic is that in this range, the projection $\P\cH_\kappa \to \cM_g$ should ``behave like a fibration''. Although this is far from being literally true, the philosophy of the work of Shimada underlying our approach is that this should be true away from a locus of positive codimension. A precise accounting of what goes wrong where leads to Shimada's results \Cref{theorem:shimada} and \Cref{prop:montriv}, which together provide the technical toolkit we use, allowing us to obtain descriptions of fundamental groups that fit into short exact sequences like those for a fibration.

\para{A counterexample in high codimension} It is also quite interesting to compare our results with the limited set of things that are known for strata of high codimension. A result of Looijenga-Mondello \cite{lm} computes the orbifold fundamental group of the stratum component $\cH_4^{odd}$ - this is the stratum in genus $3$ consisting of differentials with a single zero of multiplicity $4$ on a non-hyperelliptic curve. They find that
\[
\pi_1^{orb}(\cH_4^{odd}) \cong A(E_6)/Z,
\]
where $A(E_6)$ is the {\em Artin group of type $E_6$} and $Z$ is its center. On the other hand, a result of Wajnryb \cite{wajnryb} shows that under the so-called {\em standard} homomorphism $s: A(E_6) \to \Mod_3$ (where the generators of $A(E_6)$ are sent to Dehn twists about a configuration of simple closed curves on a surface of genus $3$ whose intersection pattern is the $E_6$ Dynkin diagram), there is a non-central element in the kernel of $s$. It is not hard to see that the monodromy map $\pi_1^{orb}(\cH_4^{odd}) \to \Mod_3$ induces the standard homomorphism. Taken together, this shows that the result of \Cref{theorem:main} fails to hold for $\cH_4^{odd}$. The work of Giannini \cite{giannini,giannini2} shows that the same is true for two other strata $\cH_{3,1}$ and $\cH_6^{even}$ both closely connected to Artin groups, and that moreover the kernels are large (e.g. contain nonabelian free groups). 

\para{Shimada's methods} The methods of this paper are deeply indebted to work of Shimada. In a series of papers \cite{shimada95,shimada03,shimada032} culminating in \cite{shimada}, Shimada develops a set of techniques to compute fundamental groups of algebraic varieties by using morphisms $f: X \to Y$ that behave like fiber bundles away from loci of high codimension. One of his motivating applications was to give a detailed proof of a result originally claimed by Dolgachev-Libgober \cite[p. 9]{DL} concerning the following question: let $\abs{L}$ be a very ample complete linear system of degree $d$ on a Riemann surface $C$. {\em What is the fundamental group of $\abs{L}_{red}$, the open sublocus of reduced divisors?} 

As observed by Dolgachev-Libgober, $\abs{L}_{red}$ is contained in the fiber of the Abel-Jacobi map, and so at the level of fundamental groups, $\pi_1(\abs{L}_{red})$ is a subgroup of the ``simple braid group'', group-theoretically given as the kernel of a map $\lambda_*: \Br_d(C) \to H_1(C;\Z)$; here $\Br_d(C) = \pi_1(\Conf_d(C))$ is the fundamental group of the configuration space of $d$ points on $C$, viewed as the space of all reduced divisors of degree $d$. (See the discussion in \Cref{ss:workingenv} for more details). 

The Abel-Jacobi map $\lambda: \Conf_d(C) \to \Pic^d(C)$ is not literally a fiber bundle, since the fibers $\abs{L}_{red}$ need not all be homeomorphic (nor even equidimensional). This reflects the fact that the locus of non-reduced divisors is the {\em dual variety} to $C$ as embedded in $\abs{L}^\vee$, and the topology of the dual variety is sensitive to the geometry of the embedding. If it {\em were}, the long exact sequence in homotopy would degenerate to a short exact sequence of fundamental groups, furnishing an {\em isomorphism} $\pi_1(\abs{L}_{red}) \cong \ker(\lambda_*)$. Shimada's work shows that when the degree $d$ is sufficiently large, $\lambda$ is ``close enough'' to a fibration for the expected short exact sequence to still exist. 

\para{Overview of the argument} This example is at the heart of our approach to \Cref{theorem:main}. Essentially, we carry out a version of this argument in families. In the case of the principal stratum $\P\cH_{\mathbf{1}}$ (differentials with $2g-2$ simple zeroes), the fiber over $C \in \cM_g$ of the projection map $\P\cH_{\mathbf{1}} \to \cM_g$ is the space $\abs{K_C}_{red}$ of reduced canonical divisors. In \Cref{S:global}, we use Shimada's tools to show that this projection induces a short exact sequence on fundamental groups, reducing the question of injectivity of $\rho: \pi_1^{orb}(\P\cH_{\mathbf{1}}) \to \Mod_{g,2g-2}$ to the level of the generic fiber $\abs{K_C}_{red}$. The theory of framed mapping class groups implies that injectivity at this level would follow from knowing $\pi_1(\abs{K_C}_{red}) \cong \ker(\lambda_*)$ was a simple braid group for $C$ generic. 

To do this, we again follow ideas developed by Shimada in \cite{shimada}. His methods readily show that there is a short exact sequence of the desired sort for the Abel-Jacobi map in degree $2g-2$, where the kernel subgroup is $\pi_1(\abs{L}_{red})$ for some {\em general} $L \in \Pic^{2g-2}(C)$. To show that one can replace a general $L$ with some specific line bundle $L_0$ (e.g. $K_C$), Shimada uses the connectedness of the Severi variety as established by Harris \cite{harris} to find an equisingular deformation between generic $2$-plane sections of $\abs{L}_{red}$ and $\abs{L_0}_{red}$, inducing an isomorphism at the level of the fundamental group. In the case of $L_0 = K_C$, this can be accomplished so long as the Weierstrass points of $C$ are sufficiently generic. We call this approach to obtaining results for specific fibers the {\em Shimada-Severi method}.

The extension to other strata $\P \cH_\kappa$ is accomplished by a more careful study of the fibers $\abs{K_C}_\kappa$ of $\P\cH_\kappa \to \cM_g$. The basic insight is that when the stratum has a relatively large proportion of simple zeroes, $\abs{K_C}_\kappa$ can itself be described as something close to a fiber bundle, with base a configuration space tracking the positions of the points in the {\em non}reduced subdivisor $D$; the fibers are now again describable as a locus $\abs{K_C(-D)}_{red}^\circ$ of reduced divisors that moreover avoid the points of $D$. In \Cref{S:SESs} we carry out this analysis, obtaining a group-theoretic description of $\abs{K_C}_\kappa$ using Shimada's techniques as applied to this setup. Finally in \Cref{section:toKC}, we show that when $C$ is sufficiently general, the Shimada-Severi method can be used to relate $\abs{K_C(-D)}_{red}^\circ$ with a general $\abs{L}_{red}^\circ$. 

\para{Related work: quadratic differentials and stability conditions} Intriguingly, the counterpart of \Cref{theorem:main} is known for the principal stratum of quadratic differentials, by work of King-Qiu \cite{kingqiu}. Their interest in this space arises from a connection between quadratic differentials and {\em stability conditions}, as pioneered by Bridgeland-Smith \cite{bridgelandsmith}. While their approach is formulated in the language of stability conditions and its environs, an important step in their argument is furnished by work of Qiu-Zhou \cite{QZ} who give a simple and elegant presentation for the simple braid group (in their work, called the {\em twist braid group}). See also \cite{qiu2} for further results of this sort.

\para{Acknowledgements} The author is supported in part by NSF grant no. DMS-2338485. It is a pleasure to thank Anatoly Libgober, Eduard Looijenga, and Eric Riedl for helpful discussions. Finally, the author would like to extend his heartfelt gratitude to an anonymous referee of the paper \cite{CImonodromy} for alerting him to Shimada's work.

\subsection{Notational conventions and basic terminology}\label{SS:notation} For convenience, here we discuss the notational conventions that will be used throughout the paper. 

The symbol $\kappa$ will always denote an unordered partition 
\[
\kappa = \{k_1, \dots, k_n, k_{n+1}, \dots, k_p\} \vdash 2g-2
\]
of $2g-2$, with all $k_i \ge 1$ and $k_i \ge 2$ if and only if $1 \le i \le n$. Partitions of general integers $d$ will generally be denoted $\mu= \{k_1,\dots, k_n, k_{n+1}, \dots, k_p\}$, and will adhere to the same convention that $k_i \ge 2$ if and only if $i \le n$. The letters $n$ and $p$ will be reserved throughout for their meaning in the context of partitions. Write
\[
\mu^{nr} = \{k_1, \dots, k_n\}
\]
for the ``nonreduced subpartition'' of entries $k_i \ge 2$. We define the quantities
\[
c(\mu) = \sum_{i= 1}^n (k_i-1) \qquad \mbox{and} \qquad d(\mu) = \sum_{i = 1}^n k_i = c(\mu) + n.
\]
Note that $d(\mu)$ is {\em not} the quantity $\deg(\mu) = \sum_{i =1}^p k_i$, only the degree of the nonreduced part $\mu^{nr}$.

Throughout this paper, all divisors are taken to be effective without further comment. For a divisor
\[
D = \sum_{i =1}^p k_i P_i
\]
with all $P_i$ distinct, the {\em type} of $D$ is the associated partition 
\[
\type(D) := \{k_1, \dots, k_p\} \vdash \deg(D).
\]

Assuming as always that $k_i \ge 2$ if and only if $1 \le i \le n$, define
\[
D^{nr} = \sum_{i=1}^n k_i P_i
\]
as the subdivisor of nonreduced points. 

Let $\mu$ be a partition. We define
\[
\cM_{g,\mu} = \{(C,D) \mid C \in \cM_g,\ \type(D) = \mu\}
\]
to be the space of compact Riemann surfaces of genus $g$ together with a divisor of type $\mu$. We will extend this notation scheme to configuration spaces as well, writing $\Conf_\mu(C)$ for the space of all divisors of type $\mu$ on $C$. 

Finally, let $\Pic^d(C)$ denote the Picard group of line bundles of degree $d$ on $C$. For a line bundle $L \in \Pic^d(C)$ and a partition $\mu$ of $d$, we write
\[
\abs{L}_\mu = \{D \in \abs{L} \mid \type(D) = \mu\}
\]
for the divisors in the linear system $\abs{L}$ of type $\mu$. In particular, let $\mathbf{1}$ denote the all-$1$'s partition of $d$; then define
\[
\abs{L}_{red} := \abs{L}_{\mathbf{1}}
\]
as the open locus of reduced divisors in $\abs{L}$. 

\section{The work of Shimada}\label{S:shimada}
Here we recall the work of Shimada of \cite{shimada} that underpins our analysis. Let $X$ and $Y$ be smooth varieties and $f: X \to Y$ a dominant morphism for which the general fiber is connected. Following \cite[p. 598]{shimada}, there exists a proper Zariski closed subset $\Sigma \subset Y$ such that 
\[
f^\circ: X^\circ \to Y^\circ
\]
is a $\cC^\infty$ locally trivial fiber bundle, where
\[
Y^\circ := Y \setminus \Sigma, \qquad X^\circ := f^{-1}(Y^\circ)\qquad \mbox{and} \qquad f^\circ:= f|_{X^\circ}: X^\circ \to Y^\circ.
\]
Choose a basepoint $b \in Y^\circ$, set $F_b := f^{-1}(b)$, and choose a lifted basepoint $\tilde b \in F_b$. Let $\iota: F_b \to X$ denote the inclusion. 

Shimada's description of the kernel of $\iota_*: \pi_1(F_b, \tilde b) \to \pi_1(X,\tilde b)$ makes use of a group-theoretic construction known as the {\em Zariski van Kampen quotient}. Let $H$ be a group, and let $\Gamma$ be a subgroup of $\Aut(H)$; here $\Aut(H)$ acts on $H$ on the right. Then the subgroup 
\[
N_\Gamma:= \pair{h^{-1} h^\gamma \mid h \in H, \gamma \in \Gamma}
\]
is easily seen to be normal; define
\begin{equation}
    \label{def:ZvK}
    H \sslash \Gamma := H / N_\Gamma.
\end{equation}

Returning to the study of $f: X \to Y$, since $f^\circ: X^\circ \to Y^\circ$ is locally trivial, there is an action of $\pi_1(X^\circ, \tilde b)$ on $\pi_1(F_b, \tilde b)$. This is given structurally as a homomorphism
\[
\psi: \pi_1(X^\circ, \tilde b) \to \Aut(\pi_1(F_b, \tilde b)),
\]
which Shimada calls the {\em lifted monodromy}.
Shimada's theorem can then be stated as follows. We remark that \Cref{theorem:shimada} is a special case of \cite[Corollary 1.1]{shimada}; in general the condition $(Z')$ can be relaxed.

\para{A note on ``codimension''} Readers of a more topological predilection should be aware that all dimensions and codimensions in this paper are taken over $\C$.  

\begin{theorem}[Cf. Corollary 1.1 of \cite{shimada}]\label{theorem:shimada}
    Let $X$ and $Y$ be smooth varieties and let $f: X \to Y$ be a dominant morphism whose general fiber is connected. Suppose the following three conditions are satisfied:
    \begin{enumerate}
        \item[(C1)] The locus $\Sing(f)$ of critical points of $f$ is of codimension $\ge 2$ in $X$,
        \item[(C2)] There exists a Zariski closed subset $\Xi_0$ of $Y$ with codimension $\ge 2$ such that $F_y:= f^{-1}(y)$ is nonempty and irreducible for any $y \in Y \setminus \Xi_0$,
        \item[(Z')] $\pi_2(Y,b) = 0$.
    \end{enumerate}
    Let $i_{X*}: \pi_1(X^\circ,\tilde b) \to \pi_1(X,\tilde b)$ be induced by the inclusion map $i_X: X^\circ \into X$. Then there is an exact sequence
    \begin{equation}
        \label{eq:mainSES}
        1 \to \pi_1(F_b, \tilde b)\sslash \ker(i_{X*}) \xrightarrow{\iota_*} \pi_1(X, \tilde b) \xrightarrow{f_*} \pi_1(Y,b) \to 1.
    \end{equation}
\end{theorem}

\para{A criterion for triviality of the monodromy} In \cite[Proposition 3.36]{shimada}, Shimada gives a criterion under which the lifted monodromy group $\psi(\ker(i_{X *}))$ appearing in \eqref{eq:mainSES} is trivial. In the interest of concision and readability, we give a slightly abbreviated account of this result with some technicalities suppressed; see \cite{shimada} for the full story.

Recall $\Sigma \subset Y$ is a proper Zariski closed subset with complement $Y^\circ = Y \setminus \Sigma$ such that $f^\circ: X^\circ \to Y^\circ$ is a $\cC^\infty$ locally trivial fibration. Enumerate the components of $\Sigma$ of codimension one as $\Sigma_1, \dots, \Sigma_N$. Then the domain of the lifted monodromy group $\ker(i_{X*}: \pi_1(X^\circ) \to \pi_1(X))$ is generated by the conjugacy classes of meridians around each of the components $f^{-1}(\Sigma_i)$.  

Suppose $X$ is the complement of a reduced hypersurface $W$ in a smooth variety $\bar X$, and that $f$ is the restriction to $X$ of a {\em projective} morphism $\bar f: \bar X \to Y$. For $y \in Y$, put $\bar F_y := \bar f^{-1}(y)$, and denote by $W_y$ the {\em scheme-theoretic} intersection $\bar F_y \cap W$. Let $\Sing(\bar f) \subset \bar X$ be the Zariski closed subset of critical points of $\bar f$. 

\begin{proposition}[Cf. Proposition 3.36 of \cite{shimada}]
\label{prop:montriv}
We assume the conditions (C1) and (C2) of \Cref{theorem:shimada}. Suppose that, for a general point $y$ of $\Sigma_i$, the following criteria hold:
\begin{enumerate}
    \item[(C3)] The intersection $\bar F_y \cap \Sing(\bar f)$ is of codimension $\ge 2$ in $\bar F_y$,
    \item[(C4)] $W_y \setminus (W_y \cap \Sing(\bar f))$ is a reduced hypersurface of $\bar F_y \setminus (\bar F_y \cap \Sing(\bar f))$.
\end{enumerate}
Then the local monodromy around $f^{-1}(\Sigma_i)$ is trivial.
\end{proposition}

\section{The global fundamental group}\label{S:global}

The goal of this section is to prove \Cref{cor:globalSES}, which describes $\pi_1^{orb}(\P\cH_\kappa)$ as an extension of a mapping class group. This will be obtained from \Cref{theorem:shimada} as follows. Let 
\[
\cM_g^{nh} := \cM_g \setminus \mathcal{HE}_g 
\]
be the complement of the hyperelliptic locus, and let $Y =\cM_g^{nh\,\prime}$ be a finite cover of $\cM_g^{nh}$ that is a quasi-projective variety (i.e. where the orbifold locus in $\cM_g^{nh}$ has been resolved). Let $X = \P\cH_\kappa^{nh\,\prime}$ be the pullback of $\P\cH_\kappa$ to $\cM_g^{nh\,\prime}$, and take $f: X \to Y$ the natural projection. To verify that the conditions (C1) and (C2) hold, we will require some preliminary results, established in \Cref{SS:globalprelim}.

\subsection{Some preliminary results}\label{SS:globalprelim}
Conditions (C1) and (C2) require an understanding of the irreducibility of the general fiber, as well as information about the singular set of the morphism. We establish some results in that direction here.

\para{An irreducibility criterion}
\begin{lemma}\label{lemma:irred}
    Let $C \in \cM_g$ and $L \in \Pic^d(C)$ be given for some $d \ge g$. Let $\mu = \{k_1, \dots, k_p\}$ be a partition of $d$ with $k_i \ge 2$ for $1 \le i \le n$ and $k_i = 1$ for $n+1 \le i \le p$. Suppose that the quantity
    \[
    h^0(L(-k_1 P_1 - \dots - k_n P_n))
    \]
    is independent of $n$-tuple $(P_1,\dots,P_n) \in C^n$. Then the set
    \[
    \abs{L}_\mu = \{D \in \abs{L} \mid \type(D) = \mu \}\subset \abs{L}
    \]
    is nonempty and irreducible.
\end{lemma}
\begin{proof}
    Define the incidence correspondence
    \[
    X(L) = \{((P_1, \dots, P_n),D) \mid D \ge \sum_{i = 1}^n k_i P_i\} \subset C^n \times \abs{L};
    \]
    let $p_1: X(L) \to C^n$ and $p_2: X(L) \to \abs{L}$ denote the projections onto the respective factors. The fiber $p_1^{-1}(P_1,\dots,P_n)$ is identified with the linear system $\abs{L(-k_1P_1 - \dots - k_n P_n)}$. By hypothesis, $X(L)$ is then the total space of a bundle over $C^n$ with fibers equidimensional projective spaces; in particular $X(L)$ is irreducible. Thus its image $p_2(X(L)) \subset \abs{L}$ is also irreducible. As $\abs{L}_\mu$ is a Zariski open subset of this, it too is irreducible. 
\end{proof}

\para{Genericity} We will see in \Cref{lemma:hypsat} below that the singularities of $f: \P\cH_\kappa \to \cM_g$ are governed by curves possessing certain exceptional special divisors. Here, we study the dimension of the locus of curves in $\cM_g$ with divisors of this form.

 Define $\Xi_\kappa \subset \cM_g$ as the locus of curves $C$ such that there exists a meromorphic function $\phi$ with polar divisor $(\phi)_\infty$ of type $\kappa^{nr}$. Abusing notation, we will continue to let $\Xi_\kappa$ denote the locus of such curves even in covers and subspaces of $\cM_g$, e.g. $\cM_g^{nh\,\prime}$; the ambient space will be indicated in context. In anticipation of the following lemma, for a partition $\mu = \{k_1, \dots, k_p\}$, define $\Hur_g(\mu)$ as the Hurwitz space of branched covers $\phi: C \to \P^1$ of degree $\deg(\mu)$ with $(\phi)_\infty$ of type $\mu$, where $C$ is a compact Riemann surface of genus $g$.

A standard application of the Riemann-Hurwitz formula yields the following.
\begin{lemma}
    \label{lemma:hurdim}
    The dimension of $\Hur_g(\mu)$ is computed as
    \[
    \dim(\Hur_g(\mu)) = 2g + 2 \deg(\mu) - c(\mu) -2.
    \]
\end{lemma}

\begin{lemma}\label{lemma:xicodim}
    For $g \ge d(\kappa) + n + 1$, the locus $\Xi_\kappa \subset \cM_g$ has codimension $\ge 2$.
\end{lemma}
\begin{proof}
    A curve $C \in \cM_g$ belongs to $\Xi_{\kappa}$ if and only if there is some $\phi: C \to \P^1$ with $\phi \in \Hur_g(\kappa^{nr})$. Any such $\phi$ in fact belongs to a two-dimensional family in $\Hur_g(\kappa^{nr})$, since post-composing by an element of $\Aut(\C)$ gives another member. Thus by \Cref{lemma:hurdim},
    \[
    \dim(\Xi_\kappa) \le 2g+ 2 \deg(\kappa^{nr}) - c(\kappa) - 4 =  2g+d(\kappa)+n - 4.
    \]
    Thus $\Xi_\kappa$ will have codimension $\ge 2$ in $\cM_g$ so long as its dimension is at most $3g-5$, i.e. in the claimed range $g \ge d(\kappa) + n + 1$.
\end{proof}

\subsection{The global short exact sequence}
Here we prove the main result of this section, \Cref{cor:globalSES}.

\begin{lemma}\label{lemma:hypsat}
   For $g \ge \max\{4,d(\kappa) + n+1\}$, the projection $f: \P\cH_\kappa^{nh\,\prime} \to \cM_g^{nh\,\prime}$ satisfies the hypotheses of \Cref{theorem:shimada}. Specifically, the following assertions hold:
   \begin{enumerate}
       \item $\P\cH_\kappa^{nh\,\prime}$ and $\cM_g^{nh\,\prime}$ are smooth varieties,
       \item For $C \in \cM_g^{nh\,\prime} \setminus \Xi_\kappa$, the fiber $F_C:= f^{-1}(C)$ is nonempty and irreducible,
       \item $\Sing(f) \subset f^{-1}(\Xi_\kappa)$, and in particular has codimension $\ge 2$,
       \item $\pi_2(\cM_g^{nh\,\prime}) \cong \pi_2(\cM_g) = 0$.
   \end{enumerate}
\end{lemma}

\begin{proof}
Assertion (1) is well-known. To prove assertion (4), recall that $\pi_2(\cM_g) = 0$, and hence the same holds for any unramified cover of $\cM_g$. For $g \ge 4$, the hyperelliptic locus $\mathcal{HE}_g$ is of (complex) codimension $(3g-3)-(2g-5) = g-2 \ge 2$ in $\cM_g$ and hence by transversality, $\pi_2(\cM_g \setminus \mathcal{HE}_g) \cong \pi_2(\cM_g)$.

Assertion (2) follows from the irreducibility criterion of \Cref{lemma:irred}. In order to apply this to $L = K_C$, we must show that $h^0(K_C(-k_1 P_1 - \dots - k_n P_n))$ is constant. By Riemann-Roch, we can instead consider $h^0(k_1 P_1 + \dots + k_n P_n)$. By the definition of $\Xi_\kappa$, since $C \not\in \Xi_\kappa$, this quantity is constant, equal to $1$. 

To prove (3), we globalize the incidence correspondence techniques of \Cref{lemma:irred}. Let $\cM_{g,\vec n}$ denote the moduli space of Riemann surfaces of genus $g$ with $n$ {\em ordered} marked points, let $\cM_{g, \vec n}^\circ$ denote the complement of the hyperelliptic locus and the preimage of $\Xi_\kappa$ in $\cM_{g,\vec n}$, and finally define $\cM_{g, \vec n}^{\circ\prime}$ as a finite unramified cover of $\cM_{g, \vec n}^{\circ}$ with empty orbifold locus. Extend this notation scheme, defining $\cM_g^{\circ \prime}$ as the complement of $\Xi_\kappa$ in $\cM_g^{nh\,\prime}$. Let $\P\Omega_{g,\vec n}^{\circ\prime}$ denote the projectivized Hodge bundle over $\cM_{g,\vec n}^{\circ\prime}$.

Define
\[
\cX_\kappa = \{(C,(P_1,\dots, P_n),D) \mid (C, (P_1, \dots, P_n)) \in \cM_{g,\vec n}^{\circ\prime}, D \in \abs{K_C}, D \ge \sum_{i=1}^n k_i P_i\} \subset \P\Omega_{g,\vec n}^{\circ\prime}.
\]

By the argument of \Cref{lemma:irred}, the projection $p: \P\Omega_{g,\vec n}^{\circ \prime} \to \cM_{g,\vec n}^{\circ \prime}$ restricts to realize $\cX_\kappa$ as the total space of a locally trivial fiber bundle with projective space fibers. In particular, as a map of smooth manifolds, $p$ is a smooth submersion. We will show that $f: \P\cH_\kappa ' \to \cM_g^{nh\,\prime}$ is a submersion over any point in $\cM_g^{nh\,\prime}\setminus \Xi_\kappa$. This will show that $f(\Sing(f)) \subset \Xi_\kappa \subset \cM_g^{nh\,\prime}$. As $\Xi_\kappa \subset \cM_g^{nh\,\prime}$ is a locus of codimension $\ge 2$, this will prove assertion (3). To do so, consider the following commutative diagram:
\[
\xymatrix{
\cX_\kappa \ar[r] \ar_{p}[d] & \overline{\P\cH_\kappa^{\circ \prime}} \subset \P\Omega_g^{\circ \prime} \ar[d]^{\bar f}\\
\cM_{g,\vec n}^{\circ \prime} \ar[r]        & \cM_g^{\circ\prime}.
}
\]
In the upper-right corner, $\P\Omega_g^{\circ \prime}$ denotes the projectivized Hodge bundle over $\cM_g^{\circ \prime}$, and $\overline{\P\cH_\kappa^{\circ \prime}}$ denotes the closure of the projectivized $\kappa$ stratum inside this. Both the left and bottom arrows are projection maps for fiber bundles, and hence are submersions. By commutativity, it follows that $\bar f$ is likewise a submersion, proving the claim. 
\end{proof}

For a Riemann surface $C$, recall that $\abs{K_C}_\kappa \subset \abs{K_C}$ denotes the canonical divisors in the stratum $\P\cH_\kappa$.

\begin{corollary}\label{cor:globalSES}
For $g \ge \max\{4,d(\kappa) + n+1\}$, there is a short exact sequence
\begin{equation}\label{eq:firstSES}
    1 \to \pi_1(\abs{K_C}_\kappa)\sslash K \to \pi_1^{orb}(\P\cH_\kappa) \to \Mod_g \to 1,
\end{equation}
where $K$ is a certain subgroup of $\Aut(\pi_1(\abs{K_C}_\kappa))$.
\end{corollary}

\begin{proof}
    By \Cref{lemma:hypsat}, the projection $f: \P\cH_\kappa^{nh\,\prime} \to \cM_g^{nh\,\prime}$ satisfies the hypotheses of \Cref{theorem:shimada}, so that there is a sequence
    \[
    1 \to \pi_1(\abs{K_C}_\kappa)\sslash K \to \pi_1(\P\cH_\kappa^{nh\,\prime}) \to \pi_1(\cM_g^{nh\,\prime}) \to 1,
    \]
    for some subgroup $K \le \Aut(\pi_1(\abs{K_C}_\kappa))$. As $g \ge 4$, the hyperelliptic locus (of dimension $2g-1$) is of codimension at least $2$, so that $\pi_1(\cM_g^{nh,\prime}) \cong \pi_1(\cM_g')$ is a finite-index subgroup of $\Mod_g$ (here $\cM_g'$ denotes the corresponding cover of $\cM_g$ with the hyperelliptic locus filled back in). Analogously, $\pi_1(\P\cH_\kappa^{nh\,\prime}) \cong \pi_1(\P\cH_\kappa')$ is isomorphic to a finite-index subgroup of $\pi_1^{orb}(\P\cH_\kappa)$. By \cite[Theorem B]{strata2}\footnote{The statement of \cite[Theorem B]{strata2} requires $g \ge 5$, but for strata for which $\gcd\{k_1, \dots, k_{p}\} = 1$ (such as we consider here), it is easy to extend the result to the regime $g \ge 3$, since the target of the monodromy map here is the ordinary mapping class group, for which Dehn twist generating sets are of course well-known. In any event, other parts of the proof of \Cref{theorem:main} require $g \ge 7$, so this is a moot point.}, the map $\pi_1(\P\cH_\kappa^{nh\,\prime}) \to \pi_1(\cM_g^{nh\,\prime})$ extends to a surjection 
    \[
    \pi_1^{orb}(\P\cH_\kappa) \onto \Mod_g.
    \]
    Thus \eqref{eq:firstSES} extends to the claimed short exact sequence
    \[
    1 \to \pi_1(\abs{K_C}_\kappa)\sslash K \to \pi_1^{orb}(\P\cH_\kappa) \to \Mod_g \to 1.
    \]
\end{proof}

\section{Linear systems and simple braid groups}\label{S:SESs}

In this section and the next, we take up a study of the structure of the fiber $\abs{K_C}_\kappa$ of the map $f: \P\cH_\kappa \to \cM_g$. To understand its fundamental group, it will be necessary to understand analogous spaces of divisors $\abs{L}_\kappa$ for general line bundles of degree $2g-2$. In this section, we establish two main structural results about such $\pi_1(\abs{L}_\kappa)$, \Cref{lemma:SESlocal,prop:LredSES}; both concern the relationship between $\pi_1(\abs{L}_\kappa)$ and certain subgroups of surface braid groups. Then in \Cref{section:toKC}, we show how to specialize these results, passing from a general line bundle of degree $2g-2$ to the canonical bundle.

\subsection{The Abel-Jacobi map}\label{ss:workingenv} 

We begin with a recollection of the Abel-Jacobi map. To avoid giving the impression that the results here are specific to degree $2g-2$, we discuss things here for an arbitrary degree $d$ and an arbitrary partition $\mu$ of $d$; be aware that we will specialize to $d=2g-2$ and a partition $\kappa \vdash 2g-2$ in the sequel.

\subsubsection{Abel-Jacobi basics} Recall that the {\em $d^{th}$ symmetric power} 
\[
\Sym^d(C) := C^d / S_d
\]
of a Riemann surface $C$ is defined as the quotient of $C^d$ by the action of the symmetric group permuting the factors, and that $\Sym^d(C)$ is naturally identified with the space of (effective) divisors on $C$ of degree $d$. Recall also from the discussion in \Cref{SS:notation} that for a partition $\mu = \{k_1, \dots, k_p\}$ of $d$, the space $\Conf_\mu(C)$ denotes the space of divisors of type $\mu$; it is a finite cover of the ordinary configuration space $\Conf_p(C)$. 

Recall the existence of the {\em Abel-Jacobi map}
\begin{align*}
    \bar \lambda: \Sym^d(C) &\to \Pic^d(C)\\
    D&\mapsto \cO(D);
\end{align*}
this has fiber $\bar \lambda^{-1}(L) = \abs{L}$. For any partition $\mu$ of $d$, there is a specialization
\[
\lambda_\mu: \Conf_\mu(C) \to \Pic^d(C);
\]
the fiber here is given by
\[
\lambda_\mu^{-1}(L) = \abs{L}_\mu.
\]

\subsubsection{Abel-Jacobi, relative to a divisor}\label{SSS:AJrel} The space $\abs{L}_\mu$ can in turn be understood in the context of an auxiliary configuration space. Consider the map 
    \[
    \pi: \Conf_\mu (C) \to \Conf_{\mu^{nr}}(C),
    \]
    projecting a divisor onto its non-reduced portion. For any $L \in \Pic^{d}(C)$, the fiber of the restriction of $\pi$ to $\abs{L}_\mu \subset \Conf_\mu(C)$ over some $D \in \Conf_{\mu^{nr}}(C)$ is then the space 
    \[
    \abs{L(-D)}_{red}^\circ := \abs{L(-D)}_{red}\setminus H_D
    \]
    of {\em reduced} effective divisors in the linear system $L(-D)$ that also miss the points of $D$; here, we have defined
    \[
    H_D = \bigcup_{i = 1}^n H_{P_i} \subset \abs{L}
    \]
    to be the union of the hyperplanes defining the condition of vanishing at each $P_i \in D$. When $L(-D)$ is very ample, 
    \[
    \abs{L(-D)}_{red} = \abs{L(-D)} \setminus C^\vee,
    \]
    where $C^\vee$ is the dual variety to $C\subset\abs{L(-D)}^\vee$. For notational convenience, define
    \begin{equation}
    \label{eq:deltaLD}
        \Delta_{L,D} = C^\vee \cup H_{P_1} \cup \dots \cup H_{P_n},
    \end{equation}
    so that
    \[
    \abs{L(-D)}_{red}^\circ = \abs{L(-D)} \setminus \Delta_{L,D}.
    \]

There is a further specialization of the Abel-Jacobi map in our present context. Suppose, as usual, that the parts $k_1, \dots, k_p$ of $\mu$ satisfy $k_i \ge 2$ if and only if $i \le n$. Define
\begin{align*}
    \lambda_{\mu,D}: \Conf_{p-n}(C \setminus D) &\to \Pic^{d}(C)\\
    E& \mapsto \cO(E+D).
\end{align*}
Note that the fiber over $L \in \Pic^{d}(C)$ is $\lambda_{\mu,D}^{-1}(L) = \abs{L(-D)}_{red}^\circ$. This extends to a morphism of smooth projective varieties
    \begin{align*}
        \bar \lambda_{\mu,D}: \Sym^{p-n}(C) &\to \Pic^{d}(C)
    \end{align*}
    by the same formula. There is a factorization of $\bar \lambda_{\mu,D}$, via
    \[
        \Sym^{p-n}(C) \xrightarrow{\bar\lambda} \Pic^{p-n}(C) \xrightarrow{\otimes \cO(D)} \Pic^{d}(C),
    \]
    with the first map the usual Abel-Jacobi map and the second the isomorphism given by the twist $L \mapsto L(D)$. 

\subsubsection{Configuration spaces and (simple) braid groups} The fundamental group of a configuration space is a braid group. We adopt the notational convention 
\[
\Br_\bullet(S) := \pi_1(\Conf_\bullet(S)),
\]
where $S$ is a surface and $\bullet$ is either an integer or a partition. The kernel of $\lambda_{\mu,*}: \Br_{\mu}(C) \to H_1(C;\Z)$ is known as a {\em simple braid group}, and has been encountered and studied by many authors \cite{DL,shimada03,walker,QZ}; the kernel of $\lambda_{\mu, D,*}: \Br_{p-n}(C\setminus D) \to H_1(C;\Z)$ will also be called a simple braid group.

\subsection{Statement of main results}

Here we present the main results of the section. We will then assemble some preliminary results in \Cref{SS:dimbounds} before giving the proofs in \Cref{SS:proofs}.

\begin{proposition}\label{lemma:SESlocal}
        Let $g \ge d(\kappa) + 3$. Let $L \in \Pic^{2g-2}(C)$ be given such that $L(-D)$ is very ample for all $D \in \Conf_{\kappa^{nr}}(C)$ and such that $h^0(K_C-L+\cO(D))$ is constant, independent of $D$. Then there is a short exact sequence
        \begin{equation}\label{eqn:SESlocal}
            1 \to \pi_1(\abs{L(-D)}_{red}^\circ) \to \pi_1(\abs{L}_\kappa) \to \Br_{\kappa^{nr}}(C) \to 1,
        \end{equation}
        where $D \in \Conf_{\kappa^{nr}}(C)$ is a general point. In particular, for $g \ge d(\kappa) + n + 4$ and $C \in \cM_g$ sufficiently general, this holds for $L = K_C$. 
    \end{proposition}

    \begin{proposition}\label{prop:LredSES}
    Let $g \ge d(\kappa) + 6$, and let $L \in \Pic^{2g-2}(C)$ and $D \in \Conf_{\kappa^{nr}}(C)$ be general. Then the sequence
    \[
    1 \to \pi_1 (\abs{L(-D)}_{red}^\circ) \to \Br_{p-n}(C \setminus D) \xrightarrow{\lambda_{\kappa,D,*}} H_1(C;\Z) \to 1
    \]
    is exact.
\end{proposition}

\subsection{Some dimension bounds}\label{SS:dimbounds}

The main results require some bounds on the dimensions of certain loci of divisors and line bundles, which we collect here. Recall that a divisor $D$ on $C$ is {\em special} if $h^0(K_C(-D)) \ge 0$, i.e. if there exists some canonical divisor $(\omega)$ vanishing at the points of $D$. Also recall that a line bundle $L$  is {\em very ample} if the equality $h^0(L(-P-Q)) = h^0(L) - 2$ holds for all pairs of points $P,Q \in C$.

\begin{lemma}\label{lemma:spdim}
    Let $C$ be a compact Riemann surface of genus $g$. Let $\SpSym^d(C) \subset \Sym^d(C)$ denote the locus of special effective divisors of degree $d$, and likewise let $\SpPic^d(C) \subset \Pic^d(C)$ denote the locus of special line bundles of degree $d$. Then $\dim(\SpSym^d(C)) \le g-1$ and hence $\codim(\SpSym^d(C)) \ge d-g+1$, and $\dim(\SpPic^d(C)) \le 2g-2-d$, so $\codim(\SpPic^d(C)) \ge d-g+2$. 
\end{lemma}
\begin{proof}
    The canonical linear system $\abs{K_C}$ has dimension $g-1$, and therefore the set of subdivisors of some canonical divisor of degree $d$ likewise has dimension at most $g-1$, proving the first assertion. If some $D \in \Sym^d(C)$ is special, then so is any linearly equivalent $E \sim D$. By Riemann-Roch, the dimension of a special linear system of degree $d$ is at least $d - g + 1$. 
    This implies that the dimension of the set of special {\em line bundles} drops by at least this much under the Abel-Jacobi map $\bar \lambda: \Sym^d(C) \to \Pic^d(C)$. This proves the second assertion.
\end{proof}

        \begin{lemma}\label{lemma:kcminusDva}
        Let $g \ge d(\kappa) + n + 4$, and let $C \in \cM_g$ be general. Then for all $D \in \Conf_{\kappa^{nr}}(C)$, $K_C(-D)$ is very ample and $h^0(D) = 1$.
    \end{lemma}
    \begin{proof}
        By Riemann-Roch, the very ampleness condition $h^0(K_C(-D-P-Q)) = h^0(K_C(-D))-2$ will hold so long as $h^0(D+P+Q) = 1$ for all $D \in \Conf_{\kappa^{nr}}(C)$ and $P, Q \in C$. Let $\kappa' = \kappa^{nr}\cup \{1,1\}$, so that $D + P + Q$ is (generically) a divisor of type $\kappa'$. Note that $\deg(\kappa') = \deg(\kappa) + 2$ and that $c(\kappa') = c(\kappa)$. By \Cref{lemma:hurdim}, 
        \[
        \dim(\Hur_g(\kappa')) = 2g + d(\kappa) +n + 2.
        \]
        If $C \in \cM_g$ has some $h^0(D+P+Q) \ge 2$, it therefore accounts for a family of dimension at least $2$ in $\Hur_g(\kappa')$, so that the dimension of the locus of such curves is at most $2g+ d(\kappa) + n$. Thus for $2g+ d(\kappa) + n \le 3g-4$, a general curve $C$ has no such divisors, equivalently for $g \ge d(\kappa) + n + 4$. If $h^0(D+P+Q) = 1$ for all $D, P, Q$, then necessarily also $h^0(D) = 1$, so that the second condition follows from the first.
    \end{proof}

We next recall the following results appearing in \cite{shimada}.

    \begin{proposition}[Cf. \cite{shimada}, Proposition 5.1]\label{prop:51shimada}
    Suppose that $g \le d \le 2g-2$. As always, let $\bar \lambda: \Sym^d(C) \to \Pic^d(C)$ denote the Abel-Jacobi map. The following assertions hold:
    \begin{enumerate}
        \item $\Sing(\bar \lambda) = \bar\lambda^{-1}(\bar\lambda(\Sing(\bar \lambda)))$.
        \item $\dim \Sing(\bar \lambda) \le g-1$ and $\dim \bar \lambda(\Sing(\bar \lambda)) \le 2g -2 - d$.
    \end{enumerate}
\end{proposition}

\begin{proposition}[Cf. \cite{shimada}, Proposition 5.4.2]\label{prop:54shimada}
    Suppose that $d \ge g+4$. Then there exists a Zariski closed subset $\Xi \subset \Pic^d(C)$ of codimension $\ge 2$ such that $\abs{L}$ is very ample for any $L \in \Pic^d(C) \setminus \Xi$.
\end{proposition}

\subsection{Proof of main results}\label{SS:proofs}
  \begin{proof}[Proof of \Cref{lemma:SESlocal}]
        We show that $p_L:= \abs{L}_\kappa \to \Conf_{\kappa^{nr}}(C)$ satisfies the criteria of \Cref{theorem:shimada} and \Cref{prop:montriv}. $\Conf_{\kappa^{nr}}(C)$ is well-known to have vanishing higher homotopy groups, so that condition (Z') holds. To show that the remaining hypotheses hold, define
        \[
        \bar{\abs{L}}_\kappa := \{(D,E) \mid E \ge D \} \subset \Conf_{\kappa^{nr}}(C) \times \abs{L},
        \]
        and define
        \[
        \bar{p_L}: \bar{\abs{L}}_\kappa \to \Conf_{\kappa^{nr}}(C)
        \]
        by 
        \[
        \bar{p_L}(D,E) = D.
        \]
        Note that $\bar{p_L}$ is projective with fiber $\bar{p_L}^{-1}(D) = \abs{L(-D)}$, that $\abs{L}_\kappa$ embeds into $\bar{\abs{L}}_\kappa$ via $E \mapsto (E^{nr},E)$, and that $\bar{p_L}$ restricts to $p_L$ on this locus. Following the discussion of \Cref{SSS:AJrel}, note that 
        \[
        W_L:= \bar{\abs{L}}_\kappa \setminus \abs{L}_\kappa
        \]
        is a hypersurface, with one component consisting of the nonreduced divisors in $\abs{L(-D)}$ and additional hyperplane components consisting of divisors in $\abs{L(-D)}$ supported at the points of $D$.

        By Riemann-Roch and the hypothesis that $h^0(K_C-L+\cO(D))$ is constant, the fibers of $\bar{p_L}$ are projective spaces of constant dimension $g-2-d(\kappa) + h^0(K_C-L+\cO(D)) \ge 1$, thereby realizing $\bar{\abs{L}}_\kappa$ as the total space of a fiber bundle over $\Conf_{\kappa^{nr}}(C)$. In particular, $\Sing(\bar{p_L})$ is empty, showing (C1) of \Cref{theorem:shimada} and (C3) of \Cref{prop:montriv}. As the fiber $\abs{L(-D)}_{red}^\circ$ is a Zariski open set in $\abs{L(-D)}$, it is in particular nonempty and irreducible, so that (C2) holds. 
        
        To see that the scheme-theoretic intersection $W_L \cap \abs{L(-D)}$ is necessarily reduced, we first observe that since $\abs{L(-D)}$ is very ample by hypothesis, the reduced scheme structure on the intersection is $\Delta_{L,D}$ as defined in \eqref{eq:deltaLD}. As is well-known (see e.g. \cite[Remark 5.6]{shimada}), the degree of the dual variety $C^\vee \subset \Delta_{L,D}$ depends only on $g$ and $\deg(L(-D))$, and likewise the degree of each remaining component $H_{P_i}$ for $i \ge 1$ is uniformly $1$. Thus the reduced scheme structure on $W_L \cap \abs{L(-D)}$ is of constant degree, and we conclude that $W_L \cap \abs{L(-D)}$ must be reduced for all $D \in \Conf_{\kappa^{nr}}(C)$. This establishes (C4). Thus all hypotheses for \Cref{theorem:shimada} and \Cref{prop:montriv} hold, establishing the existence of the sequence \eqref{eqn:SESlocal} for general $L$.

        In the case $L = K_C$, \Cref{lemma:kcminusDva} shows that for $g \ge d(\kappa) + n + 4$, for $C$ sufficiently general, $K_C(-D)$ is very ample and $h^0(D) = 1$ for all $D \in \Conf_{\kappa^{nr}}(C)$, so that the hypotheses of \Cref{lemma:SESlocal} are satisfied. 
    \end{proof}

\begin{proof}[Proof of \Cref{prop:LredSES}]
    We apply \Cref{theorem:shimada} and \Cref{prop:montriv} to $\lambda_{\kappa,D}$ and its extension $\bar \lambda_{\kappa,D}$. As a first comment, note that $\pi_2(\Pic^{2g-2}(C)) = 0$, so that condition (Z') of \Cref{theorem:shimada} holds. 
    
    We analyze the singular set $\Sing(\bar \lambda_{\kappa,D})$. Since $\otimes \cO(D): \Pic^{p-n}(C) \to \Pic^{2g-2}(C)$ is an isomorphism, $\Sing(\bar \lambda_{\kappa,D}) = \Sing(\bar \lambda)$. 
    Recall that $p -n = 2g-2 - d(\kappa)$. Since $g \ge d(\kappa) + 6$, it follows that $g \le 2g-2-d(\kappa)$, so that \Cref{prop:51shimada} applies to show that $\dim(\Sing(\bar \lambda)) \le g-1$. Again since $g \ge d(\kappa) + 6$, the inequality $g-1 \le (2g-2-d(\kappa)) - 2$ holds, showing that $\Sing(\bar \lambda) \subset \Sym^{2g-2-d(\kappa)}(C)$ has codimension $\ge 2$, as required for condition (C1) of \Cref{theorem:shimada}.

    As the fiber $\lambda_{\kappa,D}^{-1}(L) = \abs{L(-D)}_{red}^\circ$ is a nonempty Zariski open subset of the projective space $\abs{L(-D)}$ of dimension $\ge g-2-d(\kappa) \ge 1$, it is irreducible, so that condition (C2) of \Cref{theorem:shimada} holds. Thus \Cref{theorem:shimada} applies to show that there is a short exact sequence
    \[
    1 \to \pi_1(\abs{L(-D)}_{red}^\circ) \sslash K_L \to \Br_{n-p}(C \setminus D) \to H_1(C;\Z) \to 1
    \]
    for some monodromy group $K_L \le \Aut(\pi_1(\abs{L(-D)}_{red}^\circ))$. 

    We will show that $K_L$ is trivial by appealing to \Cref{prop:montriv}. Recall that this requires certain conditions (C3) and (C4) to be met for a general fiber over some unspecified locus $\Sigma_i \subset \Pic^{2g-2}(C)$ of codimension one. While we can't give an explicit description of $\Sigma_i$, we will be able to verify that the conditions are met for general fibers, which will suffice.

    Condition (C3) stipulates that the intersection $\bar \lambda_{\kappa,D}^{-1}(L) \cap \Sing(\bar \lambda_{\kappa,D})$ be of codimension $\ge 2$ in $\bar \lambda_{\kappa,D}^{-1}(L)$ for a general $L \in \Sigma_i$. By \Cref{prop:51shimada}.1, so long as $L(-D) \in \Pic^{2g-2-d(\kappa)}(C)$ is nonspecial, the fiber $\bar{\lambda}_{\kappa,D}^{-1}(L)$ is disjoint from $\Sing(\bar \lambda_{\kappa,D})$. By \Cref{prop:51shimada}.2, the locus of singular fibers has dimension $\le 2g-2-(p-n) = d(\kappa)$ and hence codimension $\ge g-d(\kappa)$. Since $g \ge d(\kappa) + 6$, a general fiber in any locus $\Sigma_i \subset \Pic^{2g-2}(C)$ of codimension one will therefore be disjoint from $\Sing(\bar \lambda_{\kappa,D})$.

    Condition (C4) stipulates that the scheme-theoretic intersection $\Delta \cap \abs{L(-D)}$, with $\Delta := \Sym^{p-n}(C) \setminus \Conf_{p-n}(C \setminus D)$,
    be reduced for a general $L \in \Sigma_i$. The argument of \Cref{lemma:SESlocal} shows that this holds so long as $\abs{L(-D)}$ is very ample. According to \Cref{prop:54shimada}, so long as $2g-2-d(\kappa) \ge g + 4$, i.e. $g \ge d(\kappa) + 6$, then there is a Zariski closed subset $\Xi \subset \Pic^{2g-2-d(\kappa)}(C)$ of codimension $\ge 2$ so that $\abs{M}$ is very ample for any $M \in \Pic^{2g-2-d(\kappa)}(C) \setminus \Xi$. Let $\Xi' \subset \Pic^{2g-2}(C)$ be the image of $\Xi$ under the isomorphism $\otimes \cO(D): \Pic^{2g-2-d(\kappa)}(C) \to \Pic^{2g-2}(C)$. Then for $g \ge d(\kappa) + 6$ and any $\Sigma_i \subset \Pic^{2g-2}(C)$ of codimension one, a general $L \in \Sigma_i$ will not lie in $\Xi'$, so that $L(-D)$ is very ample and consequently $\Delta \cap \abs{L(-D)}$ is reduced.
\end{proof}

\section{Specializing to the canonical bundle}\label{section:toKC}
Here we show how to transfer the structural results of the previous section from a general line bundle $L$ to the canonical bundle, so long as $C$ itself is sufficiently general. The main result is \Cref{prop:pi1iso}. This will require an analysis of precisely when $C$ is sufficiently general, which is established in \Cref{prop:goodproj}.

\subsection{Existence of a good projection}
By the Zariski hyperplane section theorem, to compute the fundamental group of a hypersurface complement such as $\abs{L(-D)}_{red}^\circ$, it suffices to compute the fundamental group of a plane curve complement obtained by taking a generic $2$-plane section. In the setting of discriminant complements that we consider here, this fundamental group is controlled by the geometry of the immersion $C \to \P^2$ obtained by taking a generic projection of $\abs{L(-D)}^\vee$. For a general line bundle, this projection will have only nodes and cusps, leading Shimada to introduce the notion of {\em Pl\"ucker generality}. In our setting, the hypersurface $\abs{L(-D)} \setminus \abs{L(-D)}_{red}^\circ$ also includes hyperplanes corresponding to the points of $D$; we introduce the notion of {\em sufficiently general position} to characterize when the singularities introduced by these are generic.

\begin{definition}[Pl\"ucker general]
    Let $C \subset \P^2$ be a (reduced, irreducible) plane curve possessing only nodal singularities. $C$ is said to be {\em Pl\"ucker general} if the dual curve $C^\vee$ has only nodes and cusps as singularities.
\end{definition}

\begin{definition}[Sufficiently general position]
\label{def:verygen}
Let $C \subset \P^2$ be a plane curve. A set of points $P_1, \dots, P_n \in C$ is said to be in {\em sufficiently general position} if the following conditions are met:
\begin{enumerate}
    \item Each tangent line $\ell_{P_i}$ to $C$ at $P_i$ has multiplicity $2$ at $P_i$ and is transverse to $C$ away from $P_i$. In particular, $\ell_{P_i}$ intersects $C$ only at smooth points, including $P_i$ itself,
    \item No $\ell_{P_i}$ passes through any other point $P_j$,
    \item No three $\ell_{P_i}$ meet at a common point. 
\end{enumerate}
    
\end{definition}

Define the quantity
\[
\delta(\kappa) := \binom{2g-3-d(\kappa)}{2}-g.
\]

\begin{proposition}
    \label{prop:goodproj}
    Let $g \ge d(\kappa) +  \max\{n + 4, 7\}$. Then for $C \in \cM_g$ and $D \in \Conf_{\kappa^{nr}}(C)$ both sufficiently general, a general projection $p: \abs{K_C(-D)}^\vee \to \P^2$ has image $p(C)$ a Pl\"ucker general curve of degree $2g-2 - d(\kappa)$ with $\delta(\kappa)$ nodes, and moreover the points of $D$ on the dual curve $p(C)^\vee$ are in sufficiently general position in the sense of \Cref{def:verygen}.
\end{proposition}

\begin{proof}
    This is essentially a series of dimension bounds: we will see that under our hypotheses on $g$, every condition that obstructs genericity in the sense we require holds only for a positive-codimension set of pairs $(C,D) \in \cM_{g,\kappa^{nr}}$.

    To begin with, it is necessary for the generic $K_C(-D)$ to be very ample. By \Cref{lemma:kcminusDva}, this holds for general $C$ and all $D \in \Conf_{\kappa^{nr}}(C)$ so long as $g \ge d(\kappa)+ n + 4$. 

    We next analyze when a projection $p: \abs{K_C(-D)}^\vee \to \P^2$ fails to yield a Pl\"ucker-general curve $p(C)$. It is classical that a general projection has image a nodal curve; we will assume this going forward without further comment. By duality, a projection corresponds to a choice of subspace $\P^2 \le \abs{K_C(-D)}$. The Pl\"ucker generality of $p(C)$ (or lack thereof) thus corresponds to the absence (or presence) of canonical divisors satisfying certain properties in this choice of $\Pi \cong \P^2 \le \abs{K_C(-D)}$. 
    
    Recall that $p(C)$ is Pl\"ucker general if the dual $p(C)^\vee$ has only nodes and cusps as singularities. A node of $p(C)^\vee$ corresponds to a bitangent of $p(C)$, which in turn corresponds to a canonical divisor of the form $(\omega) = D + 2P + 2 Q + E$ in $\Pi$, with $E$ reduced and supported away from $P,Q$. Likewise, a cusp of $p(C)^\vee$ corresponds to a canonical divisor of the form $(\omega) = D+ 3P + E$ (again with $E$ reduced and not supported at $P$). Thus $p(C)$ will be Pl\"ucker general if and only if every nonreduced divisor in $\Pi$ is of one of these two forms. There are three ways for a divisor to fail these conditions:
    \begin{enumerate}
        \item $p(C)$ could have a hyperflex: we must exclude all $(\omega) \ge D+ 4P$,
        \item $p(C)$ could have a bitangent line passing through a flex: exclude all $(\omega) \ge D+ 3P + 2Q$,
        \item $p(C)$ could have a tritangent line: exclude all $(\omega) \ge D+ 2P + 2Q + 2R$.
    \end{enumerate}

    We next analyze the ways that the points $P_1,\dots,P_n$ of $D$ could fail to be in sufficiently general position on $p(C)^\vee$. We first consider the condition that the tangent lines $\ell_{P_i}$ to $p(C)^\vee$ avoid all singularities of $p(C)^\vee$. Applying duality, a tangent $\ell_{P_i}$ passing through a singularity corresponds to a divisor of the form
        \[
        (\omega) \ge D^{+} +2P + 2Q
        \]
        in the case of a bitangent, and
        \[
        (\omega) \ge D^{+} + 3P,
        \]
        in the case of a flex, where $D^{+}$ denotes a divisor with the same support as $D$ but with larger degree. 
        
        The remaining conditions required for property (1) of \Cref{def:verygen} are easily dispensed with.
        If some tangent line $\ell_{P_i}$ to $P_i$ in $p(C)^\vee$ has multiplicity $\ge 3$, then $P_i$ would correspond to some non-nodal singularity on the original $p(C)$, but these do not exist. Similarly, if such a tangent line is a bitangent on $p(C)^\vee$, then $P_i$ is a node on $p(C)$, but by perturbing the projection, we can ensure that the points $P_i$ avoid the nodes of $p(C)$. 

        Conditions (2) and (3) of \Cref{def:verygen} turn out to have a common dual form. If some $\ell_{P_i}$ passes through some other $P_j$ in $p(C)^\vee$, then dually, the tangent line to $P_j$ in $p(C)$ passes through $P_i$, which happens when $\Pi$ contains some divisor of the form $(\omega) \ge D + P_i + 2P_j$. Similarly, if three $\ell_{P_i}, \ell_{P_j}, \ell_{P_k}$ meet at a common point in the dual plane, then the points $P_i, P_j, P_k$ are collinear on $p(C)$. Both of these conditions can be avoided by excluding all divisors of the form
        \[
        (\omega) \ge D^{+3},
        \]
        where $D^{+3}$ denotes a divisor with the same support as $D$ but with $\deg(D^{+3}) \ge \deg(D) + 3$.    

        To summarize, a pair $(C,D) \in \cM_{g,\kappa^{nr}}$ is suitably generic so long as there exists some $\P^2 \subset \abs{K_C(-D)}$ avoiding all canonical divisors of the following types:
        \begin{enumerate}[label={(D\arabic*)}]
            \item $(\omega) \ge D + 4P$,
            \item $(\omega) \ge D + 3P + 2Q$,
            \item $(\omega) \ge D + 2P + 2Q + 2R$,
            \item $(\omega) \ge D^+ + 2P + 2Q$,
            \item $(\omega) \ge D^+ + 3P$,
            \item $(\omega) \ge D^{+3}$.
        \end{enumerate}
        We claim that for generic $(C,D) \in \cM_{g,\kappa^{nr}}$, divisors of each of these types form a locus of codimension $\ge 3$ in $\abs{K_C(-D)}$. Define
        \[
        X_\kappa = \{((C,D),(\omega))\mid (\omega) \ge D\} \subset \P\Omega_{g,\kappa^{nr}},
        \]
        and note that $X_\kappa$ is the closure of the image of the stratum $\P\cH_\kappa$ in the projectivized Hodge bundle $\P\Omega_{g,\kappa^{nr}}$ under the map sending $(C, (\omega))$ to $((C,(\omega)_{nr}),(\omega))$. Thus $\dim(X_\kappa) = \dim(\P\cH_\kappa) = 4g-4-c(\kappa)$. For $d(\kappa) \le g -1$, the projection $X_\kappa \to \cM_{g,\kappa^{nr}}$ surjects onto the non-hyperelliptic locus in $\cM_{g,\kappa^{nr}}$, so that the generic fiber has dimension 
        \[
        \dim(\abs{K_C(-D)}) = (4g-4-c(\kappa)) - (3g-3+n) = g-1-d(\kappa).
        \]
        
        Each of the six types of divisors above corresponds to the closure of the image of a stratum $\P\cH_{\kappa'_i}$ in $\P\Omega_{g,\kappa^{nr}}$ for one or more partitions $\kappa'_i$ (e.g. (D1) corresponds to the partition $(k_1, \dots, k_n, 4, 1 \dots, 1)$, while (D5) corresponds to the set of partitions of the form $(k_1, \dots, k_i + 1, \dots, k_n, 3, 1, \dots, 1)$). So long as $d(\kappa'_i) \le g - 1$, we can repeat the above analysis and determine the dimension of the fibers over $\cM_{g,\kappa^{nr}}$ and hence the codimension inside the fiber $\abs{K_C(-D)}$ of $X_\kappa \to \cM_{g,\kappa^{nr}}$. Doing so, we find that each type has codimension $3$. Thus, for $(C,D)$ generic, the generic $\P^2 \subset \abs{K(-D)}$ avoids all divisors of type (D1) - (D6) and hence the associated projection $p(C) \subset \P^2$ will be Pl\"ucker general and the points of $D$ will be in sufficiently general position.

        It remains only to analyze the condition $d(\kappa'_i) \le g - 1$ under which this argument is valid. Recall that $d(\kappa) = c(\kappa) + n$, i.e. the sum of the codimension of the associated stratum and the number of nonreduced points. We have already seen that $c(\kappa'_i) = c(\kappa) + 3$ for each $\kappa_i'$; now note that $n$ increases by at most $3$, in the case of (D3). We conclude that $d(\kappa_i') \le g - 1$ holds so long as $g \ge d(\kappa) + 7$.
\end{proof}

\subsection{The Shimada-Severi method}

Following \Cref{lemma:SESlocal}, under suitable hypotheses, the projection $p_L: \abs{L}_\kappa \to \Conf_{\kappa^{nr}}(C)$ realizes $\pi_1(\abs{L}_\kappa)$ as an extension of $Br_{\kappa^{nr}}(C)$; this quotient is independent of the particular $L$. Thus to show that $\pi_1(\abs{K_C}_\kappa) \cong \pi_1(\abs{L}_\kappa)$, it suffices to exhibit an isomorphism of fiber groups $\pi_1(\abs{K_C(-D)}_{red}^\circ) \cong \pi_1(\abs{L(-D)}_{red}^\circ)$.

We accomplish this using a technique we call the {\em Shimada-Severi method}, following the ideas of \cite[Proof of Theorem 1.7]{shimada}. The key idea is that by the Zariski hyperplane section theorem, the fundamental group of a hypersurface complement such as $\abs{L(-D)}_{red}^\circ$ is determined by the corresponding fundamental group of a generic plane section. Under suitable hypotheses on $L$ and $K_C$, the images will be equisingular, here consisting of the dual to a nodal plane curve together with a collection of tangent lines through general points. By Harris' resolution of the Severi problem \cite{harris}, it is possible to find an equisingular deformation connecting $\abs{L(-D)}_{red}^\circ$ and $\abs{K_C(-D)}_{red}^\circ$, thereby establishing an isomorphism of fundamental groups.

\begin{proposition}
    \label{prop:pi1iso}
    Let $g \ge d(\kappa) +  \max\{n + 4, 7\}$. Then for $C \in \cM_g$,  $L \in \Pic^{2g-2}(C)$, and $D \in \Conf_{\kappa^{nr}}(C)$ all sufficiently general, there is an isomorphism 
    \[
    \pi_1(\abs{K_C(-D)}_{red}^\circ) \cong \pi_1(\abs{L(-D)}_{red}^\circ).
    \]
\end{proposition}

\begin{proof}
Recall from the discussion of \Cref{SSS:AJrel} that when a line bundle $M(-D)$ is very ample, $\abs{M(-D)}_{red}^\circ = \abs{M(-D)} \setminus \Delta_{M,D}$ is a hypersurface complement. By the Zariski hyperplane section theorem (see \cite[Theorem 1.56]{cogolludo}, quoting \cite{hamm,GorMac}), for a sufficiently general $\Pi \cong \P^2$, there is an isomorphism
\[
\pi_1(\abs{M(-D)}_{red}^\circ) \cong \pi_1(\Pi \setminus (\Pi \cap \Delta_{M,D})).
\]
Recall that $\Delta_{M,D} = C^\vee \cup H_{P_1} \cup \dots \cup H_{P_n}$. By duality, $\Pi \cap C^\vee$ can be identified with the dual $p_{M(-D)}(C)^\vee$ to $p_{M(-D)}(C) \subset \P^2$ under the projection $p_{M(-D)}: \abs{M(-D)}^\vee \to \Pi^\vee$. The components $\Pi \cap H_{P_i}$ are dual to the points $p_M(P_i) \subset \Pi^\vee$, and so are realized as the {\em tangent lines} $\ell_{P_i}$ to $p_{M(-D)}(C)^\vee$ at the points corresponding to $P_i$.

By \Cref{prop:goodproj}, a sufficiently general projection $p_{K_C(-D)}: \abs{K_C(-D)}^\vee \to \P^2$ realizes $p_{K_C(-D)}(C)$ as a Pl\"ucker-general curve of degree $2g-2-d(\kappa)$ with $\delta(\kappa)$ nodes, and the points $P_i$ are in sufficiently general position in the sense of \Cref{def:verygen}.
A general $L \in \Pic^{2g-2}(C)$ and general $D' \in \Conf_{\kappa^{nr}}(C)$ likewise give rise to a Pl\"ucker general $p_{L(-D')}(C)$ of degree $2g-2-d(\kappa)$ and $\delta(\kappa)$ nodes. Let $S_\kappa$ denote the Severi variety of plane curves of degree $2g-2-d(\kappa)$ with $\delta(\kappa)$ nodes, and let $S_\kappa^\circ$ denote the Zariski open subset of curves that are moreover Pl\"ucker general. By \cite{harris}, $S_\kappa$, and hence $S_\kappa^\circ$, is irreducible, and so there exists an equisingular family of curves containing both $p_{K_C(-D)}(C)^\vee$ and $p_{L(-D')}(C)^\vee$. Since the tangent lines $\{\ell_{P_i}\}$ to $p_{K_C(-D)}(C)^\vee$ are in general position, this can be extended to give an equisingular deformation from $\Pi \cap \Delta_{K_C,D}$ to $\Pi \cap \Delta_{L,D'}$, exhibiting an isomorphism $\pi_1(\Pi \setminus (\Pi \cap \Delta_{K_C,D})) \cong \pi_1(\Pi \setminus (\Pi \cap \Delta_{L,D'}))$.
\end{proof}

\section{Proof of \Cref{theorem:main}}

The proof of \Cref{theorem:main} will require one last lemma, relating the structure of the braid groups associated to the various configuration spaces that have been used.
\begin{lemma}
    \label{lemma:BrSES}
   So long as $g+n \ge 1$, there is a short exact sequence
   \[
   1 \to \Br_{p-n}(C \setminus D) \to \Br_{\kappa}(C) \to \Br_{\kappa^{nr}}(C) \to 1.
   \]
\end{lemma}
\begin{proof}
    The projection $\Conf_{\kappa}(C) \to \Conf_{\kappa^{nr}}(C)$ realizes $\Conf_{\kappa}(C)$ as the total space of a fiber bundle with fiber $\Conf_{p-n}(C\setminus D)$. When $g \ge 1$ or $n \ge 1$ , then $\pi_2(\Conf_{\kappa^{nr}}(C)) = 0$, and the result follows from the long exact sequence of homotopy groups of a fiber bundle.
\end{proof}

\begin{proof}[Proof of \Cref{theorem:main}]
We first comment on the genus bound $g \ge d(\kappa) + \max \{n+4,7\}$. We will invoke four results that include some kind of nontrivial genus bound: \Cref{cor:globalSES} and \Cref{lemma:SESlocal,prop:LredSES,prop:goodproj}. \Cref{cor:globalSES} requires $g \ge \max\{4, d(\kappa) + n + 1\}$, \Cref{lemma:SESlocal} requires $g \ge d(\kappa) + 3$, \Cref{prop:LredSES} requires $g \ge d(\kappa) + 6$, and \Cref{prop:goodproj} requires $g \ge d(\kappa) + \max \{n+4,7\}$. It is clear from inspection that both $d(\kappa) + n + 4$ and $d(\kappa) + 7$ are at least as large as any of the bounds appearing in other parts of the argument. In the sequel, the standing assumption $g \ge d(\kappa) + \max \{n+4,7\}$ will hold, so that every result can be validly invoked.

Let $\rho: \P\cH_\kappa \to \cM_{g,\kappa}$ denote the inclusion. We seek to show that $\rho_*: \pi_1^{orb}(\P\cH_\kappa) \to \Mod_{g,\kappa}$ is injective. By \Cref{cor:globalSES}, there is a short exact sequence
\[
1 \to \pi_1(\abs{K_C}_\kappa)\sslash K \to \pi_1^{orb}(\P\cH_\kappa) \to \Mod_g \to 1.
\]
Combining this with the Birman exact sequence for the projection $\cM_{g,\kappa} \to \cM_g$, there is a commutative diagram
\begin{equation}
    \label{eqn:firststack}
    \xymatrix{
1 \ar[r] & \pi_1(\abs{K_C}_\kappa)\sslash K \ar[r] \ar[d]^{\rho_{C,*}} & \pi_1^{orb}(\P\cH_\kappa) \ar[r] \ar[d]^{\rho_*} &\Mod_g\ar[r] \ar@{=}[d] &1\\
1 \ar[r] & \Br_\kappa(C) \ar[r] & \Mod_{g,\kappa} \ar[r] & \Mod_g \ar[r] &1
}
\end{equation}
To show injectivity of $\rho_*$, it therefore suffices to show injectivity of $\rho_{C,*}: \pi_1(\abs{K_C}_\kappa)\sslash K \to \Br_\kappa(C)$. For this, we will show that the inclusion 
\[
i_C: \abs{K_C}_\kappa \to \Conf_\kappa(C) 
\]
induces an injection
\[
i_{C,*}: \pi_1(\abs{K_C}_\kappa) \into \Br_\kappa(C).
\]
Note that $i_C$ is nothing more than the restriction of $\rho: \P\cH_\kappa \to \cM_{g,\kappa}$ to the fiber over $C$, so that injectivity of $i_{C,*}$ will indeed establish both injectivity of $\rho_{C,*}$ and the triviality of the monodromy group $K$. 

To see that $i_{C,*}$ is an injection, we analyze the Abel-Jacobi map
\[
\lambda_\kappa: \Conf_\kappa(C) \to \Pic^{2g-2}(C).
\]
By \Cref{lemma:SESlocal}, for $C$ sufficiently general, there is a short exact sequence
\[
1 \to \pi_1(\abs{K_C(-D)}_{red}^\circ) \to \pi_1(\abs{K_C}_\kappa) \to \Br_{\kappa^{nr}}(C) \to 1.
\]

This can be combined with the short exact sequence of \Cref{lemma:BrSES} to yield the commutative diagram
\begin{equation}
    \label{eq:secondstack}
    \xymatrix{
1 \ar[r] & \pi_1(\abs{K_C(-D)}_{red}^\circ) \ar[d]^{i_{C,D,*}} \ar[r] & \pi_1(\abs{K_C}_\kappa) \ar[r] \ar[d]^{i_{C,*}} & \Br_{\kappa^{nr}}(C) \ar[r] \ar@{=}[d] & 1\\
1 \ar[r] & \Br_{p-n}(C \setminus D) \ar[r] & \Br_{\kappa}(C) \ar[r] & \Br_{\kappa^{nr}}(C) \ar[r] & 1.
}
\end{equation}

It therefore suffices to show injectivity of the map on fibers 
\[
i_{C,D,*}: \pi_1(\abs{K_C(-D)}_{red}^\circ) \to \Br_{p-n}(C \setminus D).
\]
By \Cref{prop:LredSES}, there is a short exact sequence
\begin{equation}
\label{eq:AJSES}
    1 \to \pi_1(\abs{L(-D)}_{red}^\circ) \to \Br_{p-n}(C \setminus D) \xrightarrow{\lambda_{\kappa,D,*}} H_1(C; \Z) \to 1
\end{equation}
for a general line bundle $L$. In particular, the inclusion $\abs{L(-D)}_{red}^\circ \to \Conf_{p-n}(C\setminus D)$ induces an injection on $\pi_1$.

Since $\abs{K_C(-D)}_{red}^\circ = \lambda_{\kappa,D}^{-1}(K_C)$ is a fiber of $\lambda_{\kappa,D}$, it follows that there is a containment 
\[
\im(i_{C,D,*}) \le \pi_1(\abs{L(-D)}_{red}^\circ) \cong \ker \lambda_{\kappa,D,*}.
\]
We claim that $i_{C,D,*}$ surjects onto $\ker \lambda_{\kappa,D,*}$. By \cite[Theorem A]{strata3}, for $g \ge 5$, the map $\pi_1^{orb}(\P\cH_\kappa) \to \Mod_{g,\kappa}$ has image $\Mod_{g,\kappa}[\bar \phi] \le \Mod_{g,\kappa}$ an {\em absolute framed mapping class group}.\footnote{The result in \cite[Theorem A]{strata3} is formulated at the level of the unprojectivized stratum $\cH_\kappa$, but the map $\cH_\kappa \to \cM_{g,\kappa}$ inducing the homomorphism studied in \cite[Theorem A]{strata3} factors through $\P \cH_\kappa$.} According to \cite[Corollary 4.2]{CImonodromy}, the Birman exact sequence for the forgetful map $\Mod_{g,\kappa} \to \Mod_g$ restricts to give the short exact sequence
\[
1 \to \ker(\lambda_{\kappa,*}) \to \Mod_{g,\kappa}[\bar \phi] \to \Mod_g \to 1.
\]
Re-examining \eqref{eqn:firststack}, this shows that 
\[
i_{C,*}: \pi_1(\abs{K_C}_\kappa) \to \Br_\kappa(C)
\]
has image $\im(i_{C,*}) = \ker(\lambda_{\kappa,*})$. The sequence
\[
1 \to \Br_{p-n}(C \setminus D) \to \Br_{\kappa}(C) \to \Br_{\kappa^{nr}}(C) \to 1 
\]
of \Cref{lemma:BrSES} specializes to 
\[
1 \to \ker(\lambda_{\kappa,D,*}) \to \ker(\lambda_{\kappa,*}) \to \Br_{\kappa^{nr}}(C) \to 1;
\]
re-examining \eqref{eq:secondstack} then shows that $\im(i_{C,D,*}) = \ker \lambda_{\kappa,D,*}$ as claimed.

By \Cref{prop:pi1iso}, there is a group isomorphism 
\[
f:  \pi_1(\abs{L(-D)}_{red}^\circ) \cong \pi_1(\abs{K_C(-D)}_{red}^\circ).
\]
Altogether, then, we find that $f \circ i_{C,D,*}$ induces a surjective endomorphism of $\pi_1(\abs{K_C(-D)}_{red}^\circ)$. On the other hand, the isomorphism $f^{-1}$ and \eqref{eq:AJSES} together realize $\pi_1(\abs{K_C(-D)}_{red}^\circ)$ as a subgroup of $\Br_{p-n}(C\setminus D)$. The latter is a subgroup of the residually finite group $\Mod_{g,\kappa}$, and hence $\pi_1(\abs{K_C(-D)}_{red}^\circ)$ is likewise residually finite. Since $\abs{K_C(-D)}_{red}^\circ$ is a quasiprojective variety, its fundamental group is finitely generated. Every fintely generated residually finite group is Hopfian, so that the surjective endomorphism $f \circ i_{C,D,*}$ must be an isomorphism. Thus $i_{C,D,*}$ is an injection.
\end{proof}

\bibliography{bib}{}
\bibliographystyle{alpha}

\end{document}